\documentclass[11pt,reqno]{amsart}
\usepackage{color}
\usepackage[active]{srcltx}
\usepackage{a4wide}
\usepackage{amssymb, amsmath, mathtools}
\usepackage{graphicx}
\usepackage{float}
\usepackage[active]{srcltx}
\usepackage[ansinew]{inputenc}
\usepackage{hyperref}

\theoremstyle{plain}
\newtheorem{theorem}{Theorem}[section]
\newtheorem{maintheorem}{Theorem}
\newtheorem{lemma}[theorem]{Lemma}
\newtheorem{proposition}[theorem]{Proposition}
\newtheorem{corollary}[theorem]{Corollary}

\theoremstyle{remark}

\newtheorem{claim}{Claim}[section]

\newcommand{\field}[1]{\mathbb{#1}}
\newcommand{\RR}{\field{R}}

\newcommand{\dpt}{\displaystyle}
\newcommand{\NN}{\field{N}}

\numberwithin{equation}{section}

\usepackage[pagewise]{lineno}

\begin{document}

\title[Differential equations with pulses]{Differential equations with pulses: \\ existence and stability of periodic solutions}
\author[Alexandre Rodrigues]{Alexandre A. P. Rodrigues\\ Lisbon School of Economics and Management \\ Rua do Quelhas, 6, 1200-781 Lisboa \\Centro de Matem\'atica da Univ. do Porto \\ Rua do Campo Alegre, 687,  4169-007 Porto,  Portugal }
\address{Alexandre Rodrigues \\ Centro de Matem\'atica da Univ. do Porto \\ Rua do Campo Alegre, 687 \\ 4169-007 Porto   \\ Lisbon School of Economics and Management \\ Rua do Quelhas, 6, 1200-781 Lisboa \\ Portugal}
\email{alexandre.rodrigues@fc.up.pt}

\date{\today}

\thanks{AR was partially supported by CMUP, which is financed by national funds through FCT -- Fundação para a Ci\^encia e Tecnologia, I.P., under the project with reference UIDB/00144/2020.}


\begin{abstract}

We consider generic differential equations in $\RR$ with a finite number of hyperbolic equilibria, which are subject to $\omega$--periodic instantaneous perturbative pulses ($\omega>0$). Using the time-$
\omega$ map of the original system (without perturbation), we are able to find all periodic solutions of the perturbed system and   study their stability. 
In this article, we establish an algorithm to locate $\omega$--periodic solutions of impulsive systems  of frequency $\omega$, to study their stability and to locate Saddle-node bifurcations.
With our technique, we are able to fully characterise the asymptotic dynamics of the system under consideration.
 \end{abstract}

\maketitle
\setcounter{tocdepth}{1}

\section{Introduction}

Differential equations can be used to model the dynamics of many real-world phenomena. Several evolutionary processes are characterised by the fact that,  at certain moments of time, they experience an abrupt change of state (for instance a forest fire  or harvesting can abruptly change a landscape \cite{Shuai_2007}, medication \cite{Huang_2012}, vaccination \cite{Onofrio, Shi_Chen, Shulgin}). These processes are subject to short-term perturbations whose duration is negligible in comparison with the duration of the process. Thus, it is natural to assume that these perturbations are instantaneous. Such systems can be studied with dynamical systems with discontinuous trajectories \cite{Samoilenko_livro, Simeonov_paper88}.

We refer to these systems as \emph{impulsive differential equations} or \emph{differential equations with pulses}: systems of differential equations   coupled with a discrete map to capture the change in state, the so called  ``impulse''. The impulsive condition can be either time or state dependent \cite{Hernandez, Millman60}. The theory of impulsive differential equations has been introduced in the sixties by Mil'man and Mishkis \cite{Millman60}, had a fast development in the eighties and nineties \cite{Bainov_livro, Bainov97, Lakshmikantham_livro, Samoilenko_livro, Simeonov_paper86, Simeonov_paper88} and gives a good description for some real-world processes involving abrupt changes at a given sequence of times.

Impulsive equations play an important role in epidemic models with periodic vaccination  \cite{Onofrio, Shi_Chen, Shulgin} and  may arise in macroeconomics -- for example, the  Solow differential equation \cite{Emmenegger} becomes an impulsive differential equation when shocks to capital intensity are modelled with jumps. 

The most common types of   impulsive differential equations found in the mathematical literature are:
\begin{enumerate}
\item  \emph{Instantaneous impulses}: the duration of these changes is relatively short compared to the overall duration of the whole process. The sequence of times where the impulse occurs may be state dependent \cite{Bainov97, Samoilenko_livro};

\item \emph{Non-instantaneous impulses}: it is also an impulsive action, which starts at an arbitrary fixed point and remains active on a finite time interval. The authors of \cite{Hernandez} introduced this new class of abstract differential equations where the impulses are not instantaneous and explored the existence  of solutions under mild conditions;

\item \emph{Autonomous impulsive equations}:  the solution of a differential equation  $\varphi(t)$ evolves  until it hits a point within a compact subset $M$ of the phase space (if it does at all), at which the operator $I(x), x\in M$, instantaneously transfers $\varphi(t)$ to another point of the phase space \cite{Simeonov_paper88}. \\
\end{enumerate}

In the present article, motivated by impulsive SI/SIR epidemic models \cite{Onofrio, Shi_Chen, Shulgin}, we are interested in differential equations in $\RR$ under the effect if instantaneous impulses occurring at a periodic sequence of times. In \cite{Onofrio, Shulgin}, the impulse plays the role of vaccination effects in the class of Susceptibles.  

\subsection*{Novelty and structure}
In this article, we are going to consider non-linear differential equations in $\RR$ with $\omega$--periodic moments of impulsive effect. We are going to present an (easy) algorithm to find $\omega$--periodic solutions of the differential equation and classify their stability. This algorithm is useful because in $\RR$ we do not need the formalism of a Lyapunov map for impulsive differential equations discussed in Theorem 3.1 of \cite{Li} to deduce about their stability\footnote{Observe that the theory developed by \cite{Li} is valid when the periodic solution is unique.}. 

After the introduction of some basic results on impulsive differential equations in Section \ref{s:preliminaries}, we describe and motivate  the system under consideration in Section \ref{s:object of study}, where we state the main results of the paper. 
In Section \ref{s:periodic solutions} we state useful results to locate $\omega$--periodic solutions through the time-$\omega$ map associated to the unperturbed system. These results are valid for more general instantaneous impulsive differential equations. The only periodic solutions of the system under consideration are $\omega-$periodic.

In Section \ref{s:auxiliary results} we derive  technical results to study the stability of periodic solutions for impulsive systems. The proof of the main result is performed in Section \ref{s:main proof}. Section \ref{s:discussion} concludes this article. We have endeavoured to make a self contained exposition bringing together all topics related to the proofs. We have drawn illustrative figures to make the paper easily readable.

\section{Preliminaries}
\label{s:preliminaries}
  The following section contains some useful information on instantaneous impulsive differential equations may be found in \cite{Dishliev, Lakshmikantham_livro}. See also Chapter 1 of \cite{Agarwal_livro}. 

\subsection{Instantaneous impulsive differential equations}\label{object}

An impulsive differential equation is given by an ordinary differential equation coupled with a discrete map defining the ``jump'' condition. The law of evolution of the process is described by the differential equation
 $$
 \frac{dx}{dt}=f(t,x)
 $$
where $t\in \RR$, $x\in \Omega \subset \RR^n$ and $f: \RR\times \Omega\rightarrow \Omega $ is $C^1$.
The instantaneous impulse at time $t$ is defined by the map $I(t,x): \RR\times \Omega \rightarrow \Omega$ given by
$$
I(t,x)=x+I(t,x).
$$

Throughout this article, we focus on the Initial Value Problem:
 \begin{eqnarray}
 \label{prel1}
 \frac{dx}{dt}&=&f(t,x) \qquad t\neq T_k \\ 
 \Delta x(T_k) &=& I_k(x)  \qquad t=T_k,\quad  k\in \NN_0   \nonumber \\
 x(0)&=&x_0   \nonumber 
\end{eqnarray}
where $x\in \Omega\subset \RR$,  the impulse is fixed at a sequence $(T_k)_{k\in \NN_0}$  such that $T_0=0$,  $$\forall k\in \NN, \quad T_k < T_{k+1} \qquad \text{and}\qquad \lim_{k\in \NN}T_k=+\infty,$$
and, as depicted in Figure \ref{impulsive7}, the  instantaneous ``jump'' at $t=T_k$ is defined by:
$$
\Delta x(T_k):= \lim_{t\rightarrow T_k^+}\varphi(t,x)-  \lim_{t\rightarrow T_k^-}\varphi(t,x). 
$$

 \begin{figure}
\begin{center}
\includegraphics[height=4.5cm]{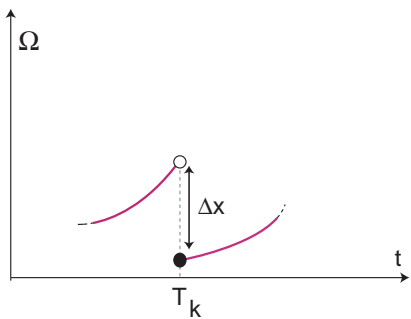}
\end{center}
\caption{\small The ``jump'' at $t=T_k$ is defined by $\Delta x(T_k):=\dpt  \lim_{t\rightarrow T_k^+}\varphi(t,x)-  \lim_{t\rightarrow T_k^-}\varphi(t,x)$, with $\Omega=\RR_0^+$.}
 \label{impulsive7}
\end{figure}

For $k\in \NN$ and $t  \in [T_k,T _{k+1})$ the solution $\varphi(t)$ of equation \eqref{prel1} satisfies the equation $\dpt \frac{dx}{dt}=f(t,x)$ and for $t\geq T_k$, $\varphi(t)$ satisfies the equality $$\lim_{t\rightarrow T_k^+}\varphi(t,x)= \lim_{t\rightarrow T_k^-}\varphi(t,x) + I_k\left(\lim_{t\rightarrow T_k^-}\varphi(t,x)\right),$$
where $x\in \Omega$.  The next result concerns the existence of a unique solution for  \eqref{prel1}.
\bigbreak
\begin{theorem}[\cite{Bainov_livro, Lakshmikantham_livro}, adapted]
Let the function $f : \RR\times \Omega \rightarrow \Omega$ be continuous in the sets $[T_k, T_{k+1}[ \times \Omega$, where $k\in \NN$. For each $k \in \NN$ and $x \in \Omega$, suppose there exists (and is finite) the limit of $f(t,y)$ as $(t,y)\rightarrow (T_k,x)$, where $ t >T_k$. Then, for each $(t_0, x_0)\in \RR\times \Omega$ there exist $\beta>t_0$ and a solution $\varphi: \, ]t_0, \beta[\rightarrow \RR$ of the initial value problem \eqref{prel1}. If $f$ is $C^1$ with respect to $x$ in $\RR\times \Omega$, the solution is unique.
\end{theorem}

  The following theorem imposes conditions where  the solution $\varphi$ may be \emph{extendable}.

\begin{theorem}[\cite{Bainov_livro, Lakshmikantham_livro}, adapted]
\label{Theorem2.1}
Let the function $f : \RR\times \Omega \rightarrow \Omega$ be continuous in the sets $[T_k, T_{k+1}[ \times \Omega$, where $k\in \NN$. For each $k \in \NN$ and $x \in \Omega$, suppose there exists (and is finite) the limit of  $f(t,y)$ as $(t,y)\rightarrow (T_k,x)$, where $ t >T_k$. If $\varphi: \, ]\alpha, \beta[ \rightarrow \RR$ is a solution of \eqref{prel1}, then the solution is extendable to the right of $\beta$ if $\beta\neq T_k$ if and only if $$\dpt \lim_{t\rightarrow \beta^-}\varphi(t)=\eta$$ and one of the following conditions holds:
\begin{enumerate}
\item $\beta\neq T_k$, for any $k\in \NN_0$ and $\eta \in \Omega$;
\item $\beta = T_k$, for some $k\in \NN_0$ and $\eta + I_k(T_k, \eta) \in \Omega$.
\end{enumerate}

\end{theorem}



Under the conditions of the Theorem \ref{Theorem2.1},  for each $(t_0 , x_0 ) \in \RR \times \Omega$, there exists a
unique solution $\varphi(t, x_0)$ of   \eqref{prel1} defined in $\RR$ (\cite{Bainov_livro, Lakshmikantham_livro}) which may be written as:
\begin{equation}
\label{general_solution}
\varphi(t, x_0)= \left\{
\begin{array}{l}
\dpt x_0 +\int_{0}^t f(s, \varphi(s, x_0)) \text{ds} + \sum_{0<T_k<t}I_k(\varphi(T_k, x_0)) \quad {\text{for}}\quad t\in [0, b[\\ \\
\dpt x_0 +\int_{0}^t f(s, \varphi(s, x_0)) \text{ds} + \sum_{t<T_k<0}I_k(\varphi(T_k, x_0))  \quad {\text{for}}\quad t\in (a, 0]\\ \\
\end{array}
\right.
\end{equation}

\subsection{Periodic solutions and stability}
\label{ss: definitions}
The following definitions have been adapted from \cite{Milev90, Simeonov_paper86, Simeonov_paper88}.
For $\omega>0$, we say that $\varphi$ is a $\omega$--periodic solution of \eqref{prel1} if and only if there exists $x_0\in \Omega$ such that 
\begin{equation}
\label{periodic_def}
\forall  t\in \RR, \quad \varphi (t, x_0)= \varphi (t+\omega, x_0).
\end{equation}
We disregard constant solutions and we consider $\omega$ as the smallest positive value for which \eqref{periodic_def} holds.
Let $x_0 \in \Omega$ be such that $\varphi(t, x_0 )$ is a  $\omega$--periodic solution of \eqref{prel1}.  We say that $\varphi(t, x_0)$  is: \\
\begin{enumerate}
\item \emph{stable} if for any neighbourhood $V$ of $x_0$, there is another neighbourhood $W\subset V$ of $x_0$ such that for all $y_0 \in W$ and for all $k\in \NN$, $\varphi(k\omega,y_0 )\in V$;\\
\item \emph{asymptotically stable} if it is stable and there exists a neighbourhood
$V$ of $x_0$ such that for any $y_0 \in V$ and $\dpt \lim_{k\rightarrow +\infty} \varphi(k\omega, y_0 )= x_0$; \\
\item \emph{unstable} if it is not stable. \\
\end{enumerate}

\section{Setting and main result}
In this section, we enumerate the main assumptions concerning the impulsive differential equation under consideration and we state the main result of this article. 
\label{s:object of study}
\subsection{Object of study}
Consider the following Initial Value Problem in $\RR_0^+$:
\begin{equation}
\label{3.1}
\left\{
\begin{array}{l}
\dot{x}=h(x)  \\ \\
 \varphi(0)=x_0 \in \RR^+_0
\end{array}
\right.
\end{equation}
whose flow is given by $\varphi(t, x)$ with $x\in \RR_0^+$ and $t\in  \RR$. The vector field $h:\RR^+_0  \rightarrow \RR^+_0$: \\ 

\begin{itemize}
\item[\textbf{(P1)}]  is $C^2$--smooth and may be written as $h(x)=Ax + H(x)$ with $A=dh(0)\neq  0$ (linear part) and $0\neq H(x)=\mathcal{O}(x^2)$, where $\mathcal{O}$ represents the usual Landau notation;\\ 
\item[\textbf{(P2)}]  has $k$ zeros, $k\in \NN$.\\
  
\end{itemize}

Without loss of generality, we are going to consider $A=dh(0)<0$ and an odd number of equilibria ($k=2n-1, n\in \NN$) -- see Figure \ref{impulsive1}. Let us denote and order the equilibria whose existence is guaranteed by \textbf{(P2)}, in the following way:
$$
0= X^s_1<X^u_1<X^s_2<X^u_2< \ldots <X^s_n.
$$ 
\bigbreak

For $\omega>0$ fixed (once and for all), define $T_k= k \omega$, $k\in \NN$, and the following Initial Value Problem (motivated by \cite{Onofrio, Shulgin}): \\
\begin{equation}\label{general}
\left\{
\begin{array}{l}
\dot{x}=h(x)  \\  \\
\dpt   \varphi(T_k,x)= (1+\lambda) \lim_{t\rightarrow T_k^-} \varphi(t,x) \qquad \text{(pulse)} \\ \\
\varphi(0, x_0)=x_0 \in \RR^+_0
\end{array}
\right.
\end{equation}
where\\

\begin{itemize}
 \item[\textbf{(P3)}] $\lambda \in \, ]-1, +\infty[$. \\
 \end{itemize}
 
\begin{figure}
\begin{center}
\includegraphics[height=4.5cm]{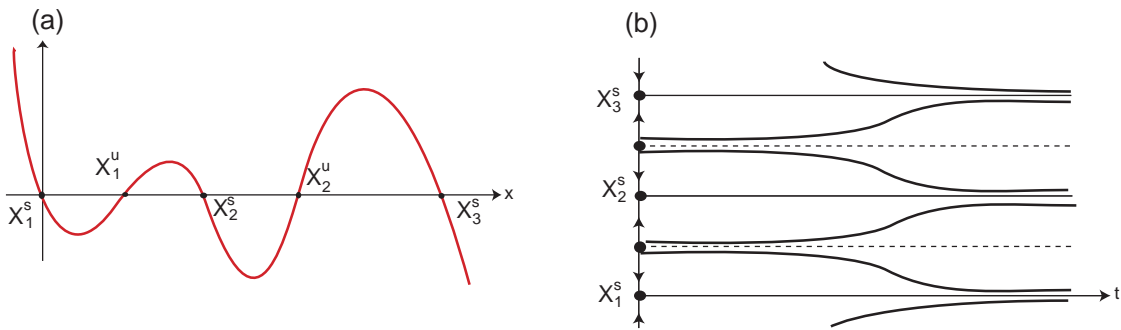}
\end{center}
\caption{\small  (a): Possible vector field $h$ of  \eqref{3.1}, where $ A<0$, $k=5$ and $n=3$ (see \textbf{(P1)--(P2)}). (b): Sketch of the graph of the solutions of \eqref{3.1} as function of $t\in \RR_0^+$. }
 \label{impulsive1}
\end{figure}



 The flow associated to the differential equation \eqref{general} is given by $\varphi_\lambda(t,x)$ where  $\varphi_0(t,x)\equiv \varphi(t,x)$,  $x\in \RR_0^+$ and $t\in \RR$. 
As depicted in Figure \ref{impulsive2}(a), let us denote by $ R_\omega(x)$ the time-$\omega$ map associated to   \eqref{3.1}:
 $$
 R_\omega(x) = \varphi_0(\omega,x).
 $$
 
 \bigbreak
 \bigbreak
 We also assume that:\\
 
 \begin{itemize}
\item[\textbf{(P4)}]   For all $j\in \{1, ..., n\}$, the map $R_\omega$ is convex for $x \in\, \,  ]X_j^s, X_j^u[ $ and concave for $x \in\, \,  ]X_j^u, X^s_{j+1}[$.\\
  \end{itemize}
 The following hypothesis, although not essential to prove the main result, simplifies the proof of our main result.  See the discussion at Section \ref{s:discussion}.\\
 \begin{itemize}
 \item[\textbf{(P5)}]  For all $j\in \{1, ..., n\}$, there exist a unique pair $(m_j, M_j) $ such that $m_j, M_j\in [X^s_j, X^s_{j+1}]$  and \\ $$\dpt \frac{R_\omega(m_j)}{m_j}=\min_{x\in [X^s_j, X^s_{j+1}]}\frac{R_\omega(x)}{x}=\beta_j<1\,, \qquad  \frac{R_\omega(M_j)}{M_j}=\max_{x\in [X^s_j, X^s_{j+1}]}\frac{R_\omega(x)}{x}=\gamma_j>1.$$ \\
 \end{itemize}

For $\lambda>-1$, the impulse map $x_0 \mapsto (1+\lambda) x_0$ is injective and increasing, which means that two solutions cannot cross/overlap after a pulse. 
 There exists a unique global solution $\varphi_\lambda(t, x_0)$ of \eqref{general} such that the operator $$(x_0, \lambda) \mapsto \varphi_\lambda(t, x_0)$$ is continuous with respect to $x_0$ and $\lambda$ (see Section 1 of \cite{Dishliev}) and may be written  explicitly  by (see \eqref{general_solution}):
$$
\varphi_\lambda(t, x_0) = x_0 +\int_{0}^t h(\varphi_0(s, x_0))ds + \sum_{ 0<T_k\leq t} \lambda  \left(\lim_{t\rightarrow T_k^-}\varphi_\lambda(t,x_0)\right), \qquad t\geq 0.
$$

 \bigbreak

\subsection{Main result}

\begin{figure}
\begin{center}
\includegraphics[height=6.5cm]{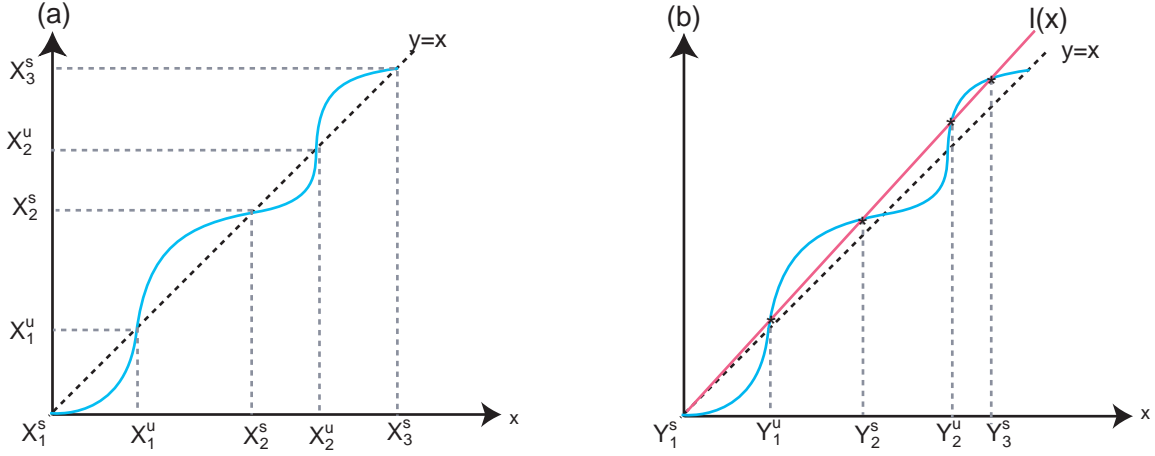}
\end{center}
\caption{\small  (a) Sketch of the map $R_\omega$ for $A<0$, $k=5$ and $x\in \RR_0^+$. (b) Illustration of the the intersection points of $R_\omega$ and $\ell$ (Item (1) of Theorem \ref{ThA}). Each intersection point, marked with $\star$, corresponds to an equilibrium of \eqref{general}.  }
 \label{impulsive2}
\end{figure}

\begin{maintheorem}
\label{ThA}
Consider the system (\ref{general}) satisfying Properties \textbf{(P1)--(P5)}. The following assertions hold:\\
\begin{enumerate}
\item The number of $\omega$--periodic orbits varies with $\lambda$ and corresponds to the number of the intersection points of the graphs of $R_\omega$ and  $\dpt \ell(x):= \frac{x}{(\lambda+1)} $.\\
\item The $\lambda$--values for which we may find  $\omega$--periodic solutions depend on $x\in \RR^+_0$ and are given  using the map $\lambda=g(x)$ where  
\begin{equation}
\label{def: g}
g(x):=\frac{x}{R_\omega(x)}-1, \qquad x>0.
\end{equation}
\item For each $x_0\in \RR^+$, it is possible to find $\lambda>-1$ for which  (\ref{general}) has a $\omega$--periodic solution associated to $x_0$.\\
 
\item Let $\varphi_\lambda $ be a $\omega$--periodic point  of \eqref{general} associated to the initial condition $x_0\in \RR$.\\
\begin{enumerate}
\item If $g'(x_0)>0$, then $\varphi_\lambda(t, x_0)$ is asymptotically stable.
\item  If $g'(x_0)<0$, then $\varphi_\lambda(t, x_0)$ is unstable.\\
\end{enumerate}

 \end{enumerate}
\end{maintheorem}
An illustration of Item (1) of Theorem \ref{ThA} may be seen in Figure \ref{impulsive2}(b) for $A<0$ and $k=5$. This result can be considered as a \emph{cinematic technique} to locate $\omega$--periodic solutions of \eqref{general} as function of $\lambda$ and $\omega$.
For $j\in \{1, ..., n\}$, solutions of  (\ref{general})  are denoted by $Y^s_j$ or $Y^u_j$ according to their stability, $s$ for stable and $u$ for unstable, where $X_1^s=Y_1^s=0$.
 The proof of Theorem \ref{ThA} is done in Section \ref{s:main proof}.\\

    The following corollary follows  from a combination of Theorem \ref{ThA}  and the analysis of $R_\omega$ for different  values of $\omega>0$, and may be useful in several applications (for instance, in epidemic models when we want to make disappear an endemic attracting periodic solution). \\

 \begin{corollary}
 \label{corol}
 Let $\omega_1>0$ and $\lambda\neq 0$.  If $Y^\star\neq 0$ is a $\omega_1$--periodic solution of \eqref{general} associated to $x_0\in \RR\backslash\{0\}$, then  there exists $\omega_2\in \, ]0, \omega_1[$ for which  $Y^\star$ no longer belong to  a periodic solution. In addition, for this $\omega_2$--value, $\varphi_\lambda(t,Y^\star)$ either converges to $X_1^s\equiv 0$ or diverges. 
 \end{corollary}
 For the sake of completeness, we perform a short proof of Corollary \ref{corol}  in Subsection \ref{Corol_prof}.
  \subsection{Digestive remarks}
We point out some remarks on the dynamics of \eqref{general} and we discuss the geometrical meaning of some hypotheses.   \\
 \begin{enumerate}
 \item First of all, note that \eqref{general} may be seen as a non-smooth perturbation of \eqref{3.1} associated to the parameter $\lambda>-1$.  If $\lambda =0$, then the differential equation  \eqref{general} does not have pulses and all solutions are smooth.\\
 \item For all $\lambda\geq -1$, $x_0=0$ is an equilibrium of \eqref{general}.\\
 \item  If $\lambda =-1$, then all solutions converge to the equilibrium $x_0=0$.\\
 \item Property \textbf{(P4)} may be rephrased as:  for all $j\in \{1, ..., n\}$, the map $R_\omega'$ is increasing for $x \in\, \,  ]X_j^s, X_j^u[ $ and decreasing for $x \in\, \,  ]X_j^u, X^s_{j+1}[$.\\
 \item The geometrical meaning of Property \textbf{(P5)} is the following: on each interval of the type $]X_j^s, X_{j+1}^u[ $, the graph of $R_\omega$ lies between the graphs of the linear maps $\beta_j x$ and $\gamma_j x$ for $j\in \{1,..., n\}$. 
 Hypothesis \textbf{(P5)}  just asks for \textbf{unicity} of $m_j$ and $M_j$; their existence is ensured by Lemma \ref{lemma_gamma} in  Section \ref{s:auxiliary results}.\\
 \item  If \textbf{(P5)} does not hold, then the proof of Theorem \ref{ThA} increases its complexity but  its statement remains qualitatively the same.  \\ 
 
 \item The case where the vector field  $h$  is a  linear map    has been studied by Milev and Bainov \cite{Milev90}. It is proved that stability is easily destroyed under small perturbations of the impulsive coefficients. In \cite{Simeonov_paper88} the authors proved an analogue of the Andronov-Witt Theorem for periodic solutions with an impulsive effect, applied to an oscillator (Example 1) and to a two-dimensional electronic system  (Example 2).
 \end{enumerate}

\section{Periodic solutions of an impulsive system}
\label{s:periodic solutions}

The following result  states a useful technique to locate $\omega$--periodic orbits in impulsive differential equations in $\RR$. \\  

\begin{proposition} 
\label{lemma_monotony}
If $\omega>0$, the following sentences hold for \eqref{general}:\\
\begin{enumerate}
\item For all $x\in \RR^+_0$  $\varphi_\lambda(\omega,x)=x$ if and only if  $\varphi_\lambda(t,x)$ is $\omega$-periodic.\\ 
\item If, for all $x\in \RR^+_0$ $\varphi_\lambda(\omega, x)>x$ then   the sequence $\left(\varphi_\lambda(T_k,x)\right)_{k\in \NN}$ is monotonic increasing.\\
\item If, for all $x\in \RR^+_0$ $\varphi_\lambda(\omega, x)<x$ then   the sequence $\left(\varphi_\lambda(T_k,x)\right)_{k\in \NN}$ is monotonic decreasing.\\

\item If $\Psi$ is a periodic solution of \eqref{general}, then $\Psi$ has period $\omega$.\\

\end{enumerate}
\end{proposition}

\begin{proof}

   \begin{enumerate}

\item Suppose that  $\varphi_\lambda(\omega, x)=x$, for some $x\in \RR_0^+$. Hence, for $t\in [0, \omega[$ the solution  exists and is well defined:
$$
\varphi_\lambda(t+\omega, x)=\varphi_\lambda(t, x).
$$
Using the impulse map, this   implies that 
 $$
\varphi_\lambda(2\omega, x)=\varphi_\lambda(\omega, x).$$
  Repeating the argument for all intervals of the form $[k\omega, (k+1)\omega)$, $k \in \NN$, it follows that for all $t \in \RR$, we have $\varphi_\lambda(t,x)=\varphi_\lambda(\omega +t,x)$ and therefore $\varphi_\lambda$ is periodic.
    \\ \\

\item Suppose that, for a given $x\in \RR_0^+$, we have $\varphi_\lambda(\omega, x)>x$. We want to prove, by induction over $k\in \NN_0$, that:

\begin{equation}
\label{induction1}
 \varphi_\lambda(T_{k+1}, x) = \varphi_\lambda((k+1)\omega, x) > \varphi_\lambda(k\omega, x)=  \varphi_\lambda(T_{k}, x).\\
\end{equation}
\bigbreak

For $k=0$, by hypothesis, we have $\varphi_\lambda(\omega, x)>x$, which is equivalent to: 
 \begin{eqnarray*}
 & & \varphi_\lambda(T_1, x)> \varphi_\lambda(0, x).\\ 
      \end{eqnarray*}
  Assuming formula \eqref{induction1} is valid for $k$, then:\\
 \begin{eqnarray*}
  \varphi_\lambda(T_{k+2}, x) &\overset{\text{Def. of $T_k$}}=&   \varphi_\lambda((k+2)\omega, x)\\ \\
&=&   \lim_{t\rightarrow(k+2)\omega^-} (1+\lambda) \varphi_\lambda(t,x) \\ \\
  &\overset{\text{Hypothesis}}>&   \lim_{t\rightarrow(k+1)\omega^-}  (1+\lambda)\varphi_\lambda(t,x) \\ \\
    &=& \varphi_\lambda((k+1)\omega, x)\\ \\
&=&    \varphi_\lambda(T_{k+1}, x)
    \end{eqnarray*} 
and we get the result. \\
\item Similar to (2).\\

\item Suppose that $\varphi_\lambda$ is a periodic solution of period $P>0$ associated to the initial condition $x_0\in \RR^+$. By definition, we have:
\begin{equation}
\label{periodic_proof1}
\forall t\in \RR, \quad \varphi_\lambda(\omega,x_0)= \varphi_\lambda(\omega+P,x_0).
\end{equation}
Therefore, we may conclude that:
\begin{eqnarray*}
\varphi_\lambda(\omega,x_0)&\overset{\text{Def.}}=& (1+\lambda) \lim_{t\rightarrow \omega^-}\varphi_\lambda(t,x_0) \\ \\
&\overset{\textbf{(P3)}}\neq& \lim_{t\rightarrow \omega^-}\varphi_\lambda(t,x_0)\\ \\
& =& \lim_{t\rightarrow (P+\omega)^-}\varphi_\lambda(t,x_0)
\end{eqnarray*} 
from where we deduce that:
$$
\varphi_\lambda(\omega+P,x_0) \neq  \lim_{t\rightarrow (P+\omega)^-}\varphi_\lambda(t,x_0),
$$
and thus $P+\omega$ is a positive multiple of $\omega$ because it should be an impulsive time.  This means that $P$ is a positive multiple of $\omega$. 
\end{enumerate}

\end{proof}

If the trajectory associate to $x$ is not constant, then either $\varphi_\lambda(\omega, x)>x$ or $\varphi_\lambda(\omega, x)<x$. In both cases, the sequence $\left(\varphi_\lambda(T_k,x)\right)_{k\in \NN}$ is monotonic. If it is bounded, then it converges either to a fixed point or to a periodic orbit.
In summary, either a non-constant solution of  \eqref{general}  converges to a $\omega$--periodic solution or it diverges.\

\section{Study of stability using the time-$\omega$ map}
\label{s:auxiliary results}
The goal of this section is to state preparatory results that are going to be needed in the proof of  Theorem \ref{ThA}. These auxiliary results characterise the maps $R_\omega$ and $g$.
Before going further, it is important to observe that for  $x_0\in \RR_0^+$ and $\lambda>-1$, we have:
\begin{equation}
\label{important1}
R_\omega(x_0)\overset{\text{Def.}}=\varphi_0(\omega, x_0)= \lim_{t\rightarrow \omega^-}\varphi_\lambda(t, x_0)= \frac{\varphi_\lambda(\omega, x_0)}{1+\lambda}.
\end{equation}
Furthermore,  the following properties are valid for $R_\omega: \RR_0^+\rightarrow \RR_0^+$: \\
\begin{enumerate}
\item It has $k$ fixed points by \textbf{(P2)};
\item It is increasing in $\RR_0^+$, bijective and continuous.
 \item If  $R_\omega(x)\gtrless x$ then for all $k \in \NN$, we have $R_\omega^k(x)\gtrless R_\omega^{k-1}(x)$, for all $k\in \NN$.\\
 \end{enumerate}

\begin{lemma}
\label{lemma5.1}
Under Hypotheses \textbf{(P1)--(P3)}, if $A= dh(0)<0$, then the following properties are valid for \eqref{general}:\\
\begin{enumerate}
 \item $g$ is continuous in $\RR^+$;\\
 \item $\dpt \lim_{x \rightarrow 0}\frac{R_\omega(x)}{x}= \exp(A \omega) $; \\
 \item $\dpt \lim_{x \rightarrow 0}g(x)= \exp(-A \omega)-1>0 $; \\
 \item if $H$ is bounded then $\dpt \lim_{x_0 \rightarrow +\infty} g(x_0)=\exp(-A \omega)-1$.\\
\end{enumerate}
\end{lemma}

\begin{proof}
\begin{enumerate}
\item  It follows immediately  from the expression of $g$ (see \eqref{def: g}), smoothness of $R_\omega$   and  $R_\omega(x)\neq 0$  for all $x\in \RR^+$. \\
\bigbreak
\item 
By \textbf{(P1)}, the vector field $h$ may be written as $h(x)= A x + H(x)$ where $H(x)=\mathcal{O}(x^2)$. Therefore, for $x_0\in \RR^+$, we may integrate \eqref{general} and we get:
$$
R_\omega(x_0)\overset{\text{Def.}}=\varphi_0(\omega, x_0)= x_0 \exp(A\omega)+\exp(A\omega)\int_0^\omega \exp(-At) H(\varphi_\lambda(t, x_0)) dt
$$
from where we deduce that (for $x_0\neq 0$):
    \begin{eqnarray}
\label{integration1}
\frac{R_\omega(x_0)}{x_0}&=&\frac{\varphi_0(\omega, x_0)}{x_0}  \nonumber\\
&=&  \exp(A\omega)+\exp(A\omega) \int_0^\omega \exp(-A t)  \frac{H(\varphi_\lambda(t, x_0))}{x_0}dt.
    \end{eqnarray}
Again by \textbf{(P1)}, since
 $ \dpt
\frac{H(  x)}{x}=  \mathcal{O}(x),
$
we may conclude that:
$$
\exp(A \omega)\int_0^\omega \exp(-A s) \frac{H(\varphi_\lambda(t, x_0))}{x_0} =  \mathcal{O}(x_0).
$$
Applying this conclusion to \eqref{integration1} we get
$$\dpt \lim_{x \rightarrow 0}  \frac{R_\omega(x)}{x}= \exp(A \omega).$$
\bigbreak
 \item From item (2) and the expression of $g$ (see \ref{def: g}), it is easy to check that:
 $$
\dpt \lim_{x \rightarrow 0}  g(x)= \exp(-A \omega)-1.
$$
\bigbreak
\item Assuming that $H$ is bounded by $M>0$ and  since\footnote{Remind that $A=dh(0)<0$.} $$\forall t\in [0, \omega], \quad \exp(A t)\leq \exp(A \omega),$$ from \eqref{integration1} we may conclude the existence of $k>0$ such that:
$$
0\leq \exp(A \omega) \int_0^\omega \exp(-A s)  \frac{H(\varphi(t, x_0))}{x_0}dt \leq   \frac{kM}{x_0}
$$
and thus $\dpt \lim_{x_0 \rightarrow +\infty}  \frac{R_\omega(x_0)}{x_0}= \exp(A \omega).$

\end{enumerate}
\end{proof}

The following corollary is a direct consequence of Lemma \ref{lemma5.1}, after the change of coordinates $x\mapsto x-X^s_j$, where $j\in \{1, ..., n\}$. 

\begin{corollary}
\label{corollary}
Under Hypotheses \textbf{(P1)--(P3)} on the  equation \eqref{general}, the following assertions hold for $j\in \{1, ...,n \}$:\\
\begin{enumerate}
  \item $\dpt \lim_{x \rightarrow X_j^s}\frac{R_\omega(x)}{x}= \exp(dh(X_j^s) \omega) $; \\
 \item $\dpt \lim_{x \rightarrow X_j^s}g(x)= \exp(-dh(X_j^s) \omega)-1>0 $. \\ \\
 
\end{enumerate}
\end{corollary}

Corollary \ref{corollary} provides the numerical value of the slope of the tangent line to $R_\omega$ at $X_j^s$. As shown in Figure \ref{impulsive3}, the next result says the graph of $R_\omega$, when restricted to $[X^s_j, X^s_{j+1}]$, lies between the graph of two linear maps.
\bigbreak

\begin{lemma}
\label{lemma_gamma}
Under Hypotheses \textbf{(P1)--(P3)}, for all $j\in \{1, ..., n\}$ there exist $\beta_j\leq 1 \leq \gamma_j$ such that $$\forall x\in [X^s_j, X^s_{j+1}], \qquad \beta_j x \leq R_\omega(x)\leq \gamma_j x.$$
 
\end{lemma}

\begin{proof}
Let us fix   $j\in \{1, ..., n\}$.  The map
$$\dpt x\mapsto \displaystyle{ \left\{\begin{array}{ll}
                                                                  \dpt  \frac{R_\omega(x)}{x} & \text{if $x\neq 0$} \\\\
                                                                   \exp(A\omega) &  \text{if $x=0$}
                                                                 \end{array} \right.}$$
  is continuous on the compact set $[X^s_j, X^s_{j+1}]$ as a consequence of  Lemma \ref{lemma5.1}. In particular, by Weierstrass' Theorem, there exist $ \beta_j\leq 1 \leq \gamma_j$  (see Figure \ref{impulsive3}) such that 
 $$\forall x\in [X^s_j, X^s_{j+1}], \qquad \beta_j  \leq \frac{R_\omega(x)}{x}\leq \gamma_j,$$
 where 
 $$
 \dpt \beta_j=\min_{x\in [X^s_j, X^s_{j+1}],}\frac{R_\omega(x)}{x} \qquad \text{and}\qquad \dpt \gamma_j=\max_{x\in [X^s_j, X^s_{j+1}],}\frac{R_\omega(x)}{x}.$$
Now the lemma follows.
 \end{proof}

\begin{figure}
\begin{center}
\includegraphics[height=7.5cm]{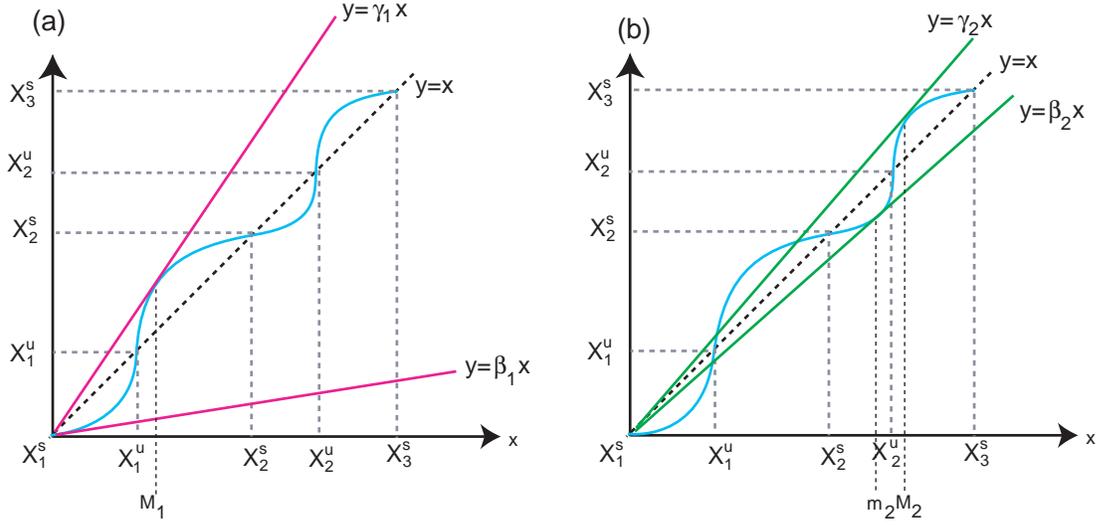}
\end{center}
\caption{\small  Illustration of Lemma \ref{lemma_gamma}: the graph of $R_\omega$, when restricted to $[X^s_j, X^s_{j+1}]$, lies between the graph of two linear maps, where $j\in \{1, ..., n\}$. (a): $j=1$. (b): $j=2$. }
 \label{impulsive3}
\end{figure}

 First of all, as a consequence of the  geometry  of the graph of $R_\omega$ stated in \textbf{(P4)} and \textbf{(P5)}, observe that $$m_1=X^s_1\equiv 0, \quad M_j \in [X^u_j, X^s_{j+1}[ \quad \text{and} \quad m_j \in [X^s_j, X^u_{j}[.$$

\begin{lemma}
\label{lemma_sign}
Under Hypotheses \textbf{(P1)--(P5)}, for all $j\in \{1, ..., n\}$, the following assertions hold:\\
 \begin{enumerate}
 \item $R_\omega'(x)> R_\omega(x) /x$ for $x\in \, ]m_j, X^u_{j}[$;\\
 \item $R_\omega'(x)> R_\omega(x) /x$ for $x\in \,  ]X^u_j, M_j[$;\\
  \item $R_\omega'(x) < R_\omega(x) /x$ for $x\in \, ]M_j, m_{j+1}[$;\\
 \item $\dpt R_\omega'(m_j)=\frac{R_\omega(m_j)}{m_j}$ and $\dpt R_\omega'(M_j)=\frac{R_\omega(M_j)}{M_j}$. \\
 \end{enumerate}
\end{lemma}

\begin{proof}
We perform the proof for $j=1$. For   $j>0$, the proof is analogous (with the necessary adaptations).\\
 \begin{enumerate}
 \item  For $x\in\,  ]m^s_1, X^u_1[$, the map $R_\omega$ is convex (by \textbf{(P4)}) and $R_\omega(0)=R_\omega(X^s_1)=0$. It is easy to check that:
 $$
R_\omega'(x)\geq \frac{R_\omega(x)-R_\omega(0)}{x-0}= \frac{R_\omega(x)}{x} 
 $$
 We want to show that, for all $x\in\,  ]m^s_1, X^u_1[$, we have:
 \begin{equation}
 \label{claim2}
 \dpt R_\omega'(x)> \frac{R_\omega(x)}{x}.
 \end{equation}
  Suppose, by contradiction, there exists $x_0\in\,  ]m^s_1, X^u_1[$ such that
   \begin{equation}
 \label{claim2a}
  \dpt R_\omega'(x_0)= \frac{R_\omega(x_0)}{x_0}.
  \end{equation}
   Since $R_\omega$ is increasing as a consequence of \textbf{(P4)}, then for all $x\in [m_1, x_0[$, we get:
\begin{equation}
\label{absurd1}
 \frac{R_\omega(x)}{x}\overset{\text{convexity}}\leq R_\omega'(x)\overset{R_\omega \text{ is decreasing}}\leq R_\omega'(x_0)\overset{\eqref{claim2a}}=\frac{R_\omega(x_0)}{x_0}.
 \end{equation}

 \begin{claim}
 \label{claim3}
 For all $x\in [m_1, x_0[$, the following equality holds:
\begin{equation}
\label{claim1}
 \frac{R_\omega(x)}{x}= \frac{R_\omega(x_0)}{x_0}
 \end{equation}
 \end{claim}
 
  \emph{Proof of Claim \ref{claim3}:}
  Suppose, by contradiction, that  Claim \ref{claim3} is false, \emph{i.e.} there exists $z  \in [m_1, x_0[$ such that $\dpt \frac{R_\omega(z)}{z}<\frac{R_\omega(x_0)}{x_0}$.     By the  Mean Value Theorem, there exists $y_0 \in\,  ]z, x_0[$ such that:
 
 \begin{eqnarray*}
 R'_\omega(y_0) &=& \frac{R_\omega(x_0)-R_\omega(z)}{x_0-z}\\ \\ 
 & > & \frac{R_\omega(x_0)-\frac{R_\omega(x_0)z}{x_0}}{x_0-z}\\ \\ 
 &= & R_\omega(x_0) \left( \frac{1-\frac{z}{x_0}}{x_0-z}\right)\\ \\ 
  &= & \frac{R_\omega(x_0)}{x_0},\\ 
  \end{eqnarray*}
 contradicting \eqref{absurd1} (note that  $m_1\leq z<y_0<x_0$). Therefore, Claim \ref{claim3} is proved.
 \bigbreak

Equality \eqref{claim1} is equivalent to 
 $
\dpt  {R_\omega(x)}=\frac{R_\omega(x_0)}{x_0} x, 
 $
 for $x\in [m_1, x_0[$. 
 In particular, near $x=0$, we may write:
 $$
 {R_\omega(x)}=\frac{R_\omega(x_0)}{x_0} x\overset{\text{Lemma \ref{lemma5.1}}}= \exp(A\omega)x
 $$
 which is the time--$\omega$ of the solution of the Malthus Law ($A=dh(X_1^s)=dh(0)$):
$$
 \left\{
\begin{array}{l}
\dot{x}=Ax \\ \\
\varphi(0, x)=x \in \RR^+_0
\end{array}
\right.
$$
However,  $\varphi_\lambda$ is the solution of \eqref{general} which is not linear (by \textbf{(P1)} we have $H(x)\neq 0$). This is a contradiction. Therefore, Equality \eqref{claim2}  is shown. \\ \\
 \item Suppose by contradiction that  $R_\omega'(x_0)\leq R_\omega(x_0) /x_0$ for a given $x_0\in\,  ]X^u_1,M_1[$. 
 Since   $R_\omega$ is concave ($\Leftrightarrow$ $R'_\omega$ is decreasing by \textbf{(P4)}), we may write that, for $x>x_0$:
 $$
 R_\omega'(x)\overset{R_\omega' \text{is decreasing }}\leq   R_\omega'(x_0) \overset{\text{contradiction}}\leq R_\omega(x_0)/x_0.
 $$
 In particular, for $y \in \, ]x_0, M_1]$, we may use the Fundamental Theorem of Calculus and we conclude that:
 
 \begin{eqnarray*}
 R_\omega(y)&=&R_\omega(x_0) +\int_{x_0}^y R_\omega'(x) dx\\
 &\leq &R_\omega(x_0) +\int_{x_0}^y \frac{R_\omega(x_0)}{x_0} dx\\
 &= &R_\omega(x_0) +  \frac{R_\omega(x_0)}{x_0}y -\frac{R_\omega(x_0)}{x_0}x_0\\
   &= & \frac{R_\omega(x_0)}{x_0}y \\
    \end{eqnarray*} 
yielding
$
R_\omega(M_1)\leq \frac{R_\omega(x_0)}{x_0}M_1.
$
 Since $R_\omega(M_1)=\gamma_1M_1$ (remind the meaning of $\gamma_1$ in Lemma \ref{lemma_gamma}) we get
 $$
\gamma_1 \leq \frac{R_\omega(x_0)}{x_0}
 $$
which means (using  \textbf{(P5)}) that $x_0=M_1$, which is a contradiction because $x_0<M_1$.  \\
 \item Similar to (2). \\
 
 \item It follows from the definition of $m_j$ and $M_j$ in the proof of Lemma \ref{lemma_gamma}.

 \end{enumerate}
\end{proof}

\bigbreak
 
\section{Proof of the main results}
\label{s:main proof}

\subsection{Proof of Theorem \ref{ThA}}

\begin{proof}
\begin{enumerate}
\item
Suppose that $\varphi(t, x_0)\in \RR^+$ is a $\omega$--periodic solution of $\eqref{general}$. On the one hand, we have $$\forall t\in \RR,\quad  \varphi_\lambda(t+\omega, x_0)= \varphi_\lambda(t, x_0),$$  which means that 
$$
\varphi_\lambda(\omega, x_0)= \varphi_\lambda(0, x_0)=x_0.
$$
 On the other hand, we may write (see \eqref{important1}):
$$
R_\omega (x_0)\overset{\text{Def.}}=\varphi_0(\omega, x_0)= \frac{\varphi_\lambda(\omega, x_0)}{1+\lambda}\overset{\varphi_\lambda \text{ is periodic}}= \frac{x_0}{1+\lambda}.
$$

Hence, the number of $\omega$--periodic orbits  is the number of the intersection points of of graphs of $R_\omega$ and  $\dpt \ell(x)= \frac{x}{\lambda+1} $.\\

\item    
If $x_0\in \RR^+$  is such that $\varphi_\lambda(t, x_0)$ is a $\omega$--periodic solution of (\ref{general}), then
$$
R_\omega (x_0)=  \frac{x_0}{1+\lambda}
$$
which is equivalent to
$$
\lambda = \frac{x_0}{R_\omega (x_0)}-1= :g(x).\\
$$
\bigbreak 
\item It is a consequence of the way the map $g$ has been constructed in (2). \\ 
\item  We prove the first case; the other is analogous. If $x_0\in \RR^+$ corresponds to a $\omega$--periodic solution associated to $\lambda^\star>-1$ then :

\begin{equation}
\label{eq_star} 
\lambda^\star = \frac{x_0}{R_\omega (x_0)}-1.
\end{equation}
 Differentiating $g$ at $x_0\neq 0$ we get:
\begin{equation}
\label{diff1}
g'(x_0)= \frac{R_\omega(x_0)-x_0R_\omega'(x_0)}{R_\omega^2(x_0)},
\end{equation}
whose signonly depends on the sign of the numerator $R_\omega(x_0)-x_0R_\omega'(x_0)$ fully characterised in Lemma \ref{lemma_sign}.
\medbreak

Suppose that $x_0$ is such that $g'(x_0)>0$.  This means that   $g$ is increasing near $x_0$ and $R_\omega'(x_0)<R_\omega(x_0)/x_0$. By Lemma \ref{lemma_sign} it means that $x\in \, ] M_j, m_{j+1}[$, for some $j\in \{1, ..., n\}$.
  In particular for $x>x_0$, we have:
 \begin{eqnarray*}
  g(x)>g(x_0) &\Leftrightarrow& \frac{x}{R_\omega(x)}-1> \frac{x_0}{R_\omega(x_0)}-1 \\ \\
  &\overset{\eqref{eq_star}}\Leftrightarrow& \frac{x}{R_\omega(x)}-1> \lambda^\star\\ \\
  &\Leftrightarrow& \frac{x}{R_\omega(x)}> 1+\lambda^\star\\ \\
   &\Leftrightarrow& {R_\omega(x)}< \frac{x}{1+\lambda^\star} \\ \\
   &\Leftrightarrow& (1+\lambda^\star) {R_\omega(x)}< x\\ \\
      &\Leftrightarrow&    \varphi_\lambda(\omega, x)<\varphi_\lambda(0, x)\\ \\
      &\overset{\text{Lemma \ref{lemma_monotony}}}\Leftrightarrow&   \text{ $\left(\varphi_\lambda(T_k,x)\right)_{k\in \NN}$ is monotonically decreasing} \\
      \end{eqnarray*} 
    
Since the sequence $\left(\varphi_\lambda(T_k,x)\right)_{k\in \NN}$ is monotonic decreasing with lower boundary $x_0$, then $\varphi_\lambda(T_k, x)$ should converge to the $\omega$--periodic solution associated to $x_0$ ($\Leftrightarrow$ $\varphi_\lambda(t, x_0)$ is asymptotically stable, by definition of Subsection \ref{ss: definitions}). The case where $g'(x_0)<0$ has a similar proof.

\end{enumerate}
\end{proof}

\subsection{Proof of Corollary \ref{corol}}
\label{Corol_prof}
In the space of $C^1$--maps endowed with the $C^1$--Whitney topology, the map $R_\omega$ is diffeotopic to the Identity. In other words,
 we have $\dpt \lim_{\omega \rightarrow 0} R_\omega =_{C_1}  Id$ where $Id$ denotes the Identity map in $\RR$.  Therefore if $Y^\star\neq 0$ is a $\omega_1$--periodic solution of \eqref{general} associated to $x_0\in \RR^+$, then it corresponds to an intersection of the graphs of  $\ell $ and $R_\omega$ (different from the origin). Since  $\dpt \lim_{\omega \rightarrow 0}\frac{x}{R_{\omega}(x)}-1=\mathbf{0}$, then the graph of $g$ converges to the null map $\mathbf{0}$, ruling out  $\lambda$--values for possible intersection.   Corollary \ref{corol} is proved.


\begin{figure}
\begin{center}
\includegraphics[height=7.5cm]{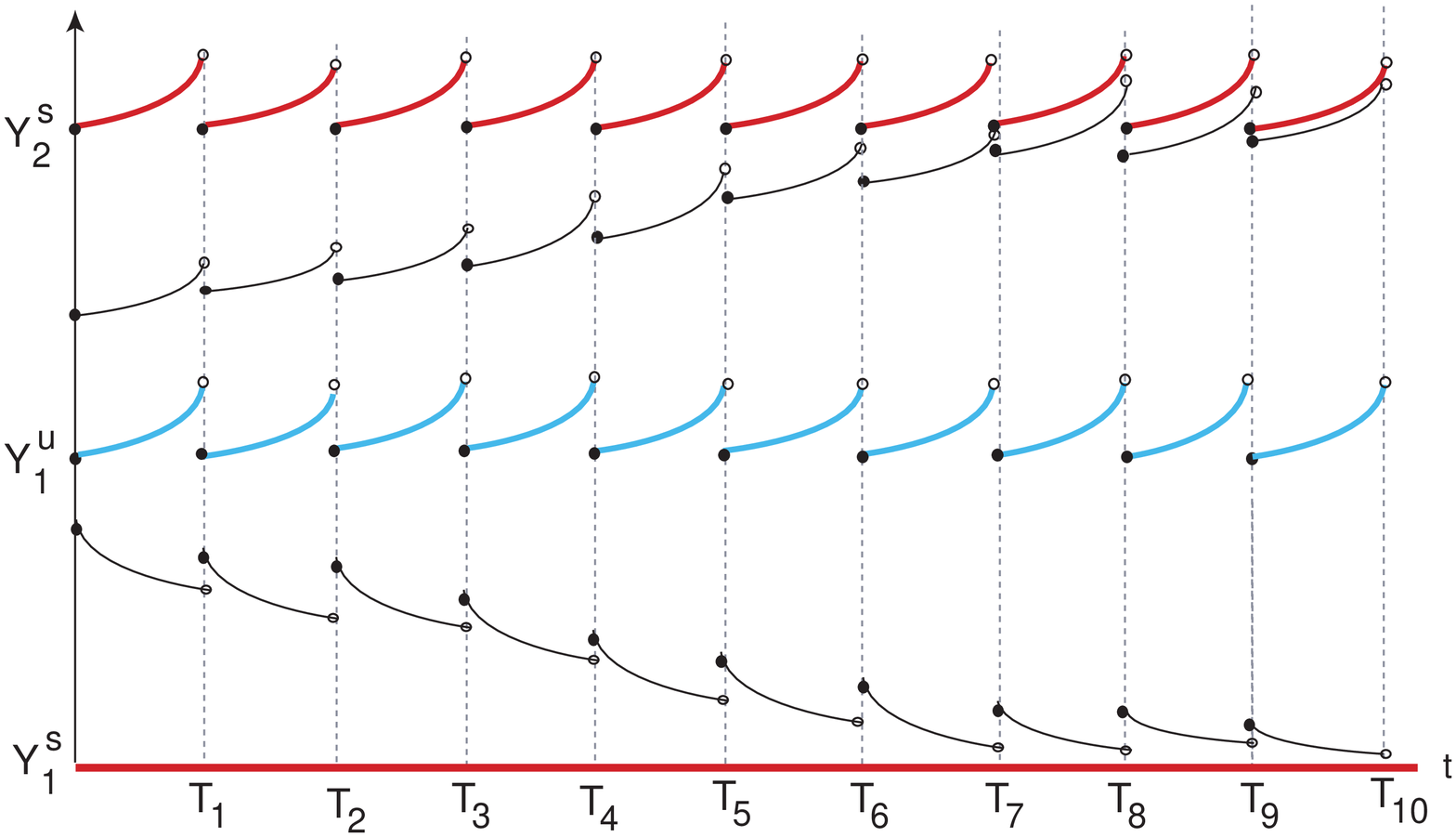}
\end{center}
\caption{\small  Sketch of three $\omega$--periodic solutions for the scenarios 2 and 4 of Table 1: two stable (blue) and one unstable (red).  }
 \label{impulsive5}
\end{figure}

 

 

\begin{figure}
\begin{center}
\includegraphics[height=18.5cm]{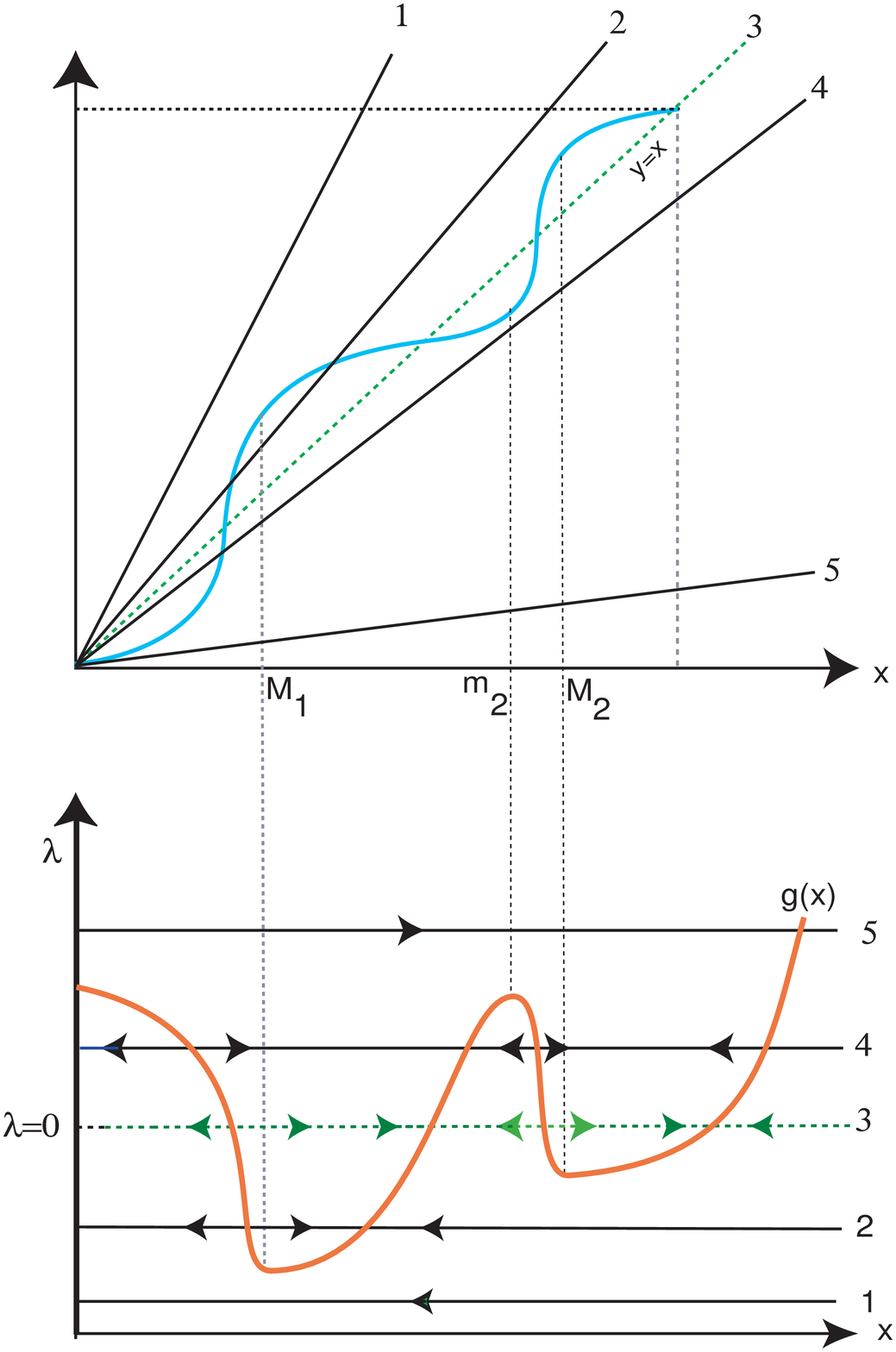}
\end{center}
\caption{\small  Geometrical interpretation of  Theorem \ref{ThA}. Possible intersections of $\ell$ and $R_\omega$; for each intersection point, we may decide about its stability by observing the monotony of $g$. For $j\in \{1, 2, 3, 4, 5\}$, the line $j$ is the graph of $\dpt \frac{x}{1+\lambda_j}$, where $-1<\lambda_j<\lambda_{j+1}$. The value of $\lambda_3$ is $0$ (its  dynamics corresponds to that of \eqref{3.1}. Arrows in the lower scheme indicate the stability of the periodic solutions. Compare the values with those of Table 1.}
 \label{impulsive4}
\end{figure}
   \begin{table}[htb]
\begin{center}
\begin{tabular}{|c|c|c|c|} \hline 
Line Reference & Interval of $\lambda$ & {Number of equilibria}  & Stability of Equilibria   \\
in Figure \ref{impulsive4}&&of \eqref{general}&of \eqref{general} (ordered) \\
\hline 
&&&\\
1&$\dpt \lambda<\frac{1}{\gamma_1}-1$ & 1&  $0=Y^s_1$ \\ &&&\\ \hline  
&&&\\ 
2& $\dpt \frac{1}{\gamma_2}-1<\lambda<\frac{1}{\gamma_1}-1$ & 3& $0=Y^s_1, Y^u_1, Y^s_2$  \\ &&&See Figure \ref{impulsive5} \\ &&& \\\hline  
&&&\\
3&$\dpt \frac{1}{\gamma_1}-1<\lambda<\frac{1}{\beta_2}-1$ & 5& $0=Y^s_1, Y^u_1, Y^s_2, Y^u_2, Y^s_3$ \\ &&&\\ \hline  
&&&\\ 
4&$\dpt \frac{1}{\gamma_1}-1<\lambda<\exp(-A\omega)-1$ & 3& $0=Y^s_1, Y^u_1, Y^s_2$   \\ &&&See Figure \ref{impulsive5} \\&&& \\\hline  
&&&\\
5&$\dpt \lambda>\exp(-A\omega)-1$ & 2&  $0=Y^u_1, Y^s_1$  \\ &&&\\ \hline  

 \hline
\end{tabular}
\end{center}
\label{notation2}
\bigskip
\caption{Qualitative dynamics of \eqref{general} for $k=5$, $n=3$, $A<0$, $\gamma_1>\gamma_2$ and $\beta_2>\exp(-A\omega)-1$, where $A=dh(0)=dh(X^s_1)$.   }
\end{table}

\section{Discussion and Final Remark}
\label{s:discussion}

In this article, we have  considered non-linear differential equations in $\RR$ with a finite number of hyperbolic equilibria, which are subject to $\omega$--periodic instantaneous impulses ($\omega>0$). 
The only non-constant recurrent solutions of our system are $\omega$--periodic ones.
 We present a cinematic algorithm to find the $\omega$--periodic solutions of the perturbed system by intersecting the graphs of two real valued maps: $\ell$ and $R_\omega$, where the latter is the time-$
\omega$ map of the unperturbed system.

In almost all cases, the map $R_\omega$ cannot be obtained explicitly. Nonetheless, it is a solvable numerical problem, from where we may derive the auxiliary map $g$; then  we are able to study the stability of the periodic solutions as a consequeence of   Theorem \ref{ThA}.
Furthermore, it suffices to determine  the initial values of the periodic solutions and their stability to conclude about the asymptotic behaviour of all solutions of \eqref{general}, as noticed in Table 1 for for $k=5$,  $A<0$, $\gamma_1>\gamma_2$ and $\beta_2>\exp(-A\omega)-1$.

 
As a consequence of the proof of Theorem \ref{ThA}, we get that if $x_0\in \RR^+$ is such that $g'(x_0)=0$, then $\varphi_\lambda(t, x_0)$ is a \emph{saddle-node bifurcation}: slight smooth perturbations on $\lambda$ give rise either to zero or two periodic solutions of different stability. In the case under consideration in Table 1 ($A<0, k=5$), this occurs when
$$
\lambda= \frac{1}{\gamma_1}-1, \quad \lambda= \frac{1}{\beta_2}-1\quad \text{and}\quad \lambda= \frac{1}{\gamma_2}-1.
$$ 
Moreover, at $\lambda= \exp(-A\omega)-1$, the origin undergoes a \emph{transcritical bifurcation}: there is an exchange of stabilities between two equilibria. The analysis of these bifurcations come directly from  Table 1 and  Figure \ref{impulsive4}.\\

  Hypotheses \textbf{(P1)--(P5)} hold for a generic system of differential equations in $\RR$, in the sense that they are valid in a \emph{residual} within the set of $C^2$ maps. 
  Hypotheses \textbf{(P4)--(P5)} simplify the proof and can be relaxed in the following way:\\
  \begin{itemize}
  \item For each $j\in \{1, ..., n\}$, the change of sign of $R'_\omega$ should occur (at least once) within $[X^s_j, X^s_{j+1}]$. Hypothesis  \textbf{(P4)} states  that it happens at $X=X^u_j$, which is not a loss of generality; another value would be possible and the proof would run along the same lines (up to the necessary changes in Lemma \ref{lemma_sign});\\
  \item   Hypothesis \textbf{(P5)} simplifies the proof of Theorem \ref{ThA} and may be relaxed. If \textbf{(P5)} is removed, then the graph of $g$ will have some intervals where it is constant. \\
  \end{itemize}
  
The natural problem about differential equations with instantaneous pulses is the  generalisation the two-dimensional results of Example 2 of \cite{Simeonov_paper88} using the techniques of \cite{Li, Shuai_2007}, a task deferred to a future work.

\end{document}